\documentclass[a4paper,12pt, oneside]{amsart}

\usepackage{amscd}
\usepackage{amsfonts}
\usepackage{amssymb}
\usepackage{amstext}
\usepackage{amsthm}
\usepackage{color}
\usepackage[nodate]{datetime}
\usepackage{dsfont}
\usepackage{graphicx}
\usepackage{hyperref}
\usepackage[latin1]{inputenc}
\usepackage{mathrsfs}
\usepackage{mathtools}
\usepackage{psfrag}
\usepackage{srcltx}
\usepackage[normalem]{ulem}
\usepackage{verbatim}
\usepackage{vruler}
\usepackage{wasysym}
\usepackage[normalem]{ulem}
\usepackage{mathrsfs}
\usepackage{stmaryrd}
\usepackage[T1]{fontenc}
\usepackage{lmodern}
\usepackage{marginnote}
\usepackage{chngcntr}
\usepackage{setspace}
\usepackage{import}
\usepackage{nameref}
\usepackage[shortlabels]{enumitem}

\newtheorem{theorem}[equation]{Theorem}

\newtheorem{corollary}[equation]{Corollary}
\newtheorem{proposition}[equation]{Proposition}
\theoremstyle{remark}
\newtheorem{remark}[equation]{Remark}

\newcommand{\E}{\mathbb{E}}
\newcommand{\N}{\mathbb{N}}
\renewcommand{\P}{\mathbb{P}}
\newcommand{\R}{\mathbb{R}}

\newcommand{\one}{{{\bf \rm 1}}}

\renewcommand{\leq}{\leqslant}
\renewcommand{\le}{\leqslant}
\renewcommand{\geq}{\geqslant}
\renewcommand{\ge}{\geqslant}

\usepackage[lmargin=27mm,rmargin=27mm,top=35mm,bottom=40mm,paperwidth=210mm,paperheight=11in]{geometry}

\def\ii{\mathbf z}
\def\jj{\mathbf z}
\def\rr{\mathbf r}

\def\ss{\mathbf s}

\def\tt{\mathbf t}
\def\uno{\mathbf 1}
\def\ff{\mathbf f}
\def\qq{\mathbf q}

\date{\mydate}

\counterwithin{equation}{section}

\begin{document}


\title{Tumor growth, $R$-positivity, Multitype branching and Quasistationarity.}
\author[A. Ferrari]{Anal\'ia Ferrari}
\address{Universidad de Buenos Aires and IMAS-CONICET, Buenos Aires, Argentina}
\email{aferrari@dm.uba.ar}
\author[P. Groisman]{Pablo Groisman}
\address{Universidad de Buenos Aires and IMAS-CONICET, Buenos Aires, Argentina}
\email{pgroisman@dm.uba.ar}
\author[K. Ravishankar]{Krishnamurthi Ravishankar}
\address{NYU-ECNU Institute of Mathematical Sciences at NYU Shanghai, 3663 Zhongshan Road North, Shanghai,
200062, China and NYU Abu Dhabi.}
\email{kr26@nyu.edu}

\begin{abstract}
Motivated by tumor growth models we establish conditions for the $R-$positivity of Markov processes and positive matrices. We then apply them to obtain the asymptotic behavior of the tumors sizes in the supercritical regime. 
\end{abstract}

\date{}
\maketitle
\section{Introduction}
\label{sec:intro}
 
Our work in this paper explores the $R-$positivity and quasistationarity for a class of models motivated by the tumor growth model studied by L. Triolo in \cite{Tr05}. His work in turn was motivated by works of \cite{IKS,S} where a continuum model of tumor growth is studied. In that model a primary tumor  starts somewhere (space is not a variable) with  one cell and grows  at a rate $h(x)$ where $h$ is taken  to be  a Gompertzian law.  That is, the number of cells of the primary tumor $x_p$ is given by the following ordinary differential equation
\[
\frac{d x_p}{dt} = h(x_p), \qquad x_p(0) =1.
\]
Here $h(x_p) =  a x_p \log (\frac{N}{x_p})$, thus asymptotically the cell size approaches $N$.
The malignant behavior is modeled by creation of one new (metastases) cell at a rate $\beta(x)$ which is an increasing function of $x$. Each new cell grows and proliferates according to the same rule as the primary tumor.
The proliferation rate $\beta(x)$ is taken to be $\beta(x) = \kappa x^r$ where $0 < r \leq 1$ and $\kappa$ is a constant called the colonization constant. The evolution of the distribution of metastases $u(x,t)$ is given by the following equation.
\[
\frac{\partial u}{\partial t}(x,t) + \frac{\partial (h u)}{\partial x}(x,t) = 0,\quad  x>1, \, t>0,
\]
with a boundary condition at $x=1$ given by  
\[
h(1) u(1,t) = \int_1^\infty \beta(x) u(x,t) dx + \beta(x_p(t)).
\]
This equation is analyzed in \cite{IKS,S}, where it is shown that $u$ increases asymptotically exponentially in $t$.

In \cite{Tr05} a birth and death model  was proposed with a view towards including the random localized (proliferation occuring at size one)  nature of cell growth as well as to include the possibility of modeling immune response  by a nonzero probability of death when the cell size is one. It is a microscopic model in which individual particles (tumor sizes) evolve independently as a birth and death process on $\N \cup\{0\}$ with $0$ being an absorbing state. The birth and death rates are chosen so that the drift matches the Gompertzian law of the continuum model. If we denote the occupation number of site $x$ by $\eta(x)$, then with rate $\sum_1^\infty \beta(x) \eta(x)$ a particle is created at site one (metastases). For such a model it is shown in \cite{Tr05} that as $t \to \infty$ the expected occupation numbers converge to 1) zero 2) a constant nonzero value or 3) diverge exponentially to infinity depending on the value of a parameter $\kappa_0$ being 1) less than one 2)  equal to one 3) greater than one. The parameter $\kappa_0$ is the expected total number particles created by one particle before being eliminated and is defined precisely below.
 
We study the asymptotic behavior of this model in the supercritical regime $\kappa_0>1$. The main tool is to identify the process with a multitype branching process and to show that its mean matrix is $R-$positive, which leads to a Kesten-Stigum type theorem for the asymptotic behavior of the distribution of tumor sizes \cite{moy,Eng1,Eng2}. In the course of the proof we provide a general criteria to establish the $R-$ positivity of positive matrices by identifying certain transformations of  them with rates matrices of absorbed Markov processes.

\subsection{The Model.} We use an interacting particle system to model our process where the particles move according to a Markov process with a countable state space $\Lambda$ and get absorbed at the state $0$. We denote $(\eta_t(x), \, x\in\Lambda)$ the number of
particles at site $x$ at time $t$.

Let $Q_0=(q(x,y), \, x,y \in \Lambda_0)$ be a rates matrix of a pure jump Markov process
on $\Lambda_0:=\Lambda \cup \{0\}$. We assume that $Q$ is irreducible, $0$ is absorbing ($q(0,y)=0$ for all $y\in \Lambda$) and denote $Q$ the restriction of $Q_0$ to $\Lambda$. We use the convention $q(x,x)=-q(x)=-\sum_{y\in \Lambda_0 \setminus \{x\}}q(x,y)$. Observe that
$q(x,0)= -\sum_{y\in \Lambda} q(x,y)$ is the absorption rate from state $x$. The process can be described as follows.
\begin{enumerate}
 \item Each particle evolves independently according to the rates matrix
$Q=(q(x,y), \, x,y \in \Lambda)$. In particular, a particle a site $x$ is absorbed at a state that we call 0 at rate $q(x,0)$.
\item Particles are created at a state that we call 1 at rate $\sum_x
\beta(x)\eta_t(x)$.
\end{enumerate}
More precisely, $(\eta_t, t\ge 0)$ is a Markov process with state space $\N^{\Lambda}$ and generator given by
\begin{equation}
 \label{generator}
 \mathcal L f (\eta) = \sum_{x\in \Lambda}\sum_{y\in \Lambda_0} q(x,y)\eta(x)[f(\eta^{x,y}) - f(\eta)] + \sum_{x \in \Lambda} \beta(x)\eta(x)[f(\eta^+)-f(\eta)].  
\end{equation}
Here
\[
 \eta^{x,y}(z) = \begin{cases}
			  \eta(z), & z\notin \{x,y\},\\
                          \eta(x)-1, & z=x\\
                          \eta(y)+1 & z=y, y \ne 0. 
                         \end{cases} \qquad 
                        \eta^+(z) = \begin{cases}
					  \eta(z), & z\ne 1,\\
					  \eta(1)+1, & z=1.
					  \end{cases}
\]
Of particular interest is the case where $Q$ is a birth and death process on $\N$ with absorption at $0$ and birth and death rates given by $\lambda_x = a x \ln (N+1)$ and $\mu_x = ax \ln (x+1)$ respectively. We denote this process by $(Y_t)_{t\ge0}$. With such a choice the drift $\lambda_x-\mu_x$ follows a Gompertzian law $g(x) = ax \ln \frac{N+1}{x+1}$. Here $x$ represents the tumor size. Also of special interest is the case $\beta(x)=\kappa x^r$ for $0<r\le 1$. In \cite{Tr05} it is proved that in this situation the process
is supercritical if and only if 
\[
\kappa_0:= \int_0^\infty \gamma_1(t)\,dt > 1.
\]
Here 
\[
 \gamma_1(t)= \sum_x \beta(x)\E_1(\one\{Y_t=x\})
\]
is the expected creation rate of a single particle. More precisely, it is proved that

\begin{enumerate}
 \item[(a)] $\kappa_0<1 \Rightarrow \E\eta_t(x)\to 0$ for all $x\in\N$.
\item[(b)] $\kappa_0=1 \Rightarrow \E\eta_t(x)$ converges to a constant.
\item[(c)] $\kappa_0>1 \Rightarrow   \E\eta_t(x)$ grows exponentially fast.
\end{enumerate}
The proof can be extended with no difficulty to our more general situation. It can also be obtained from the following observation.

We can couple the total number of particles with a single-type Galton-Watson process. For a given particle, consider the total number of children given birth by a particle during its lifetime. This is the offspring distribution. In this way the total number of particles of generation $n$ is a Galton-Watson process (but not the total number of particles at time $t$). From this observation, and by means of irreducibility, (a), (b), (c) can be extended to 

\begin{enumerate}
 \item[(d)] $\kappa_0\le 1 \Rightarrow \eta_t(x)\to 0$ for all $x\in\N$.
\item[(e)] $\kappa_0>1 \Rightarrow   \eta_t(x)$ grows exponentially fast for all $x\in \N$.
\end{enumerate}
Both statements hold almost surely.


The process can be constructed in a standard way and so, we omit the proof of its existence. 

After proving general conditions to guarantee $R-$positivity of positive matrices and the asymptotic behavior of supercritical multitype branching process we will obtain for this model the following.

\begin{theorem}
\label{main.theorem}
Under adequate assumptions on $\beta$ and $Q$, if $\kappa_0>1$ there is positive probability of non-extinction and on this event we have
\begin{equation}
\label{lim.proportions}
 \lim_{t\to \infty} \frac{\eta_t}{|\eta_t|} = \nu, \qquad \text{in  probability}.
\end{equation}
\end{theorem}
Here $\nu$ is a finite measure, which is the left eigenvector of a matrix obtained as a transformation of $Q$ (see section 3 for details).
The paper is organized as follows. In Section \ref{rpositivity} we establish two general criteria to prove $R-$positivity of nonnegative matrices. In both of them we transform the matrix to obtain a sub-Markovian operator. Next we apply a Lyapunov type criteria in the first case (Theorem \ref{Lyapunov}) and a Do\"eblin type argument in the second one, Theorem \ref{alpha.C}. In Section \ref{application} we apply these criteria to prove Theorem \ref{main.theorem}.

Similar strategies have been applied in a series of papers by N. Champagnat, D. Villemonais (\cite{champagnat-villemonaisPTRF,champagnat-villemonais, champagnat2019} among them) to establish existence of quasi-stationary distributions and uniform convergence towards them. Our techniques differ from theirs. They are probabilistic in nature and based on a theorem by Kesten, Ferrari and Mart\'inez \cite{FKM} that provides probabilistic conditions to establish $R-$positivity.

\section{Lyapunov functions, D\"oeblin conditions and $R$-positivity}
\label{rpositivity}

In this section we consider continuous or discrete time Markov chains and give conditions under which the process is $R$-positive. We will use them in Section \ref{application} to establish the asymptotic behavior of the tumor growth model. The processes in this section should not be confused with the driving process with rates $Q$ of the previous section. For that reason, in this section we use the letters $A$, $A_0$, $\tilde A$ for generators instead of $Q$.

Let $X=(X_t, t\ge 0)$ be a continuous time pure jump Markov process with rates matrix $A_0=(a(x,y), x, y \in \Lambda \cup \{0\})$. We use again the convention $a(x,x)=-\sum_{y\ne x}a(x,y)$. We assume that $X$ is absorbed at zero and denote $A$ the restriction of $A_0$ to $\Lambda$, so that $-\sum_{y\in \Lambda} a(x,y)=a(x,0)$ is the absorption rate from state $x$. For a function $V\colon \Lambda \to \R_{\ge 0}$, define the drift of $V$ at $x \in \Lambda$ by
\[
 \dot V(x) = \sum_{y \in \Lambda_0} a(x,y)V(y).
\]
We always assume $V(0)=0$ and that the sum above is well defined and finite.
\begin{theorem}[Lyapunov condition]
\label{Lyapunov}
Assume there exists $V\colon \Lambda \to \R_{\ge 0}$ with $V(x) \to \infty$ as $x\to \infty$ and such that $\lim_{x \to \infty}\frac{\dot V(x)}{V(x)} = -\infty$. Then $X$ is $R$-positive for some $0 < R \le 1$. 
\end{theorem}

\begin{remark}
This theorem has to be compared with \cite[Theorem 5.1]{champagnat-villemonais}. Our assumptions seem to be more restrictive than those on \cite{champagnat-villemonais}, however our proof can be easily modified to fit those assumptions. We prefer to write it in this way since this is a condition that can actually be checked in practice. Also, although not exactly the same, the conclusions are similar (but \cite{champagnat-villemonais} is more general and includes and exponential convergence to quasistationarity statement). We decided to include our proof since it is short and simple. It is based on ideas and the following theorem from \cite{FKM}. 
\end{remark}

\begin{theorem}
[{\cite[Theorem~1]{FKM}}]
\label{thm:h1h2h3}
Assume $P=\exp(A)$ is irreducible and call $R$ its decay parameter. Let $\tau=\inf\{ n \in \N \colon X_n=0\}$. Suppose that there exist a finite set $\mathscr U_1 \subset \Lambda$, a state $x'\in \mathscr U_1$, $\rho<R$, and a positive constant~$\kappa$ 
such that for all $x\in \mathscr U_1$ and $n\ge 0$, $\P_x(\tau>n;\, X_1 \notin \mathscr U_1 ,
\dots,X_{n} \notin \mathscr U_1 ) \le \kappa\rho^{n}$;
Then $P$ is $R$-positive and its left eigenvector~$\nu$ is summable.
\end{theorem}

\begin{remark}
The statement of \cite[Theorem 1]{FKM} is more involved since their subset $\mathscr U_1$ can be infinite. In our case, conditions (1.23) and (1.24) in \cite{FKM} follow immediately from irreducibility of $A$. The hypotheses of the theorem above correspond to what they call (1.22). In our statement $R$ is what they call $1/R$. The main ingredient of the proof of Theorem \ref{Lyapunov} is to show that this condition is in fact verified.
\end{remark}

\begin{proof}[Proof of Theorem \ref{Lyapunov}]
Let $X=(X_t, t\ge 0)$ be a process with a generator $A$. Given $x \in \Lambda$ define $g_x(t)=\E_x(V(X_t)\one\{X_t\notin \mathscr U_1\})$.  We have for $x \notin \mathscr U_1$
\[
 \frac{d}{dt}g_x(t) \le \E_x  \dot V(X_t)\one\{X_t \notin \mathscr U_1\}.
\]
Given $0<\rho<1$, we can choose $\mathscr U_1$ such that $-\dot V(x)/V(x) > -\log \rho$ if $x \notin \mathscr U_1$. Then, for such $\mathscr U_1$ we have
\[
\frac{d}{dt}g_x(t) \le (\log \rho)  \E_x  V(X_t)\one\{X_t \notin \mathscr U_1\} = (\log \rho)g_x(t).
\]
This and $g_x(0)=V(x)$ gives us $g_x(t)\le V(x) \rho^t$ and in particular $\E_x(V(X_1))/V(x) \le \rho$ for any $x\notin \mathscr U_1$ and also $\sup_x \E_x(V(X_1))/V(x) \le C$ for every $x\in \Lambda$. Given $\tilde{K}$, we can enlarge $\mathscr U_1$ if necessary to ensure that $x \notin \mathscr U_1$ implies $V(x) \ge \tilde K$. Then we have for any $x$ with $V(x)\le K$,

\begin{align*}
\P_x & \big(  X_1 \in \mathcal U_1^c  ,\dots,  X_n \in \mathcal U_1^c\big)
 \le
\frac{1}{\tilde K}\E_x\left[V(X_n) \one\{X_1 \in \mathcal U_1^c,\dots,
X_n \in \mathcal U_1^c\}\right]
\\
&
=
\frac{V(x)}{\tilde K}
\E_x\left[ \frac{V(X_1)}{V(x)}\cdots \frac{V(X_{n-1})}{V(X_{n-2})} \frac{V(X_n)}{V(X_{n-1})} \one\{X_1 \in \mathcal U_1^c,\dots, X_n \in \mathcal U_1^c\} \right ]
\\
&
\le
\E_x\left[ \frac{V(X_1)}{V(x)}\cdots \frac{V(X_{n-1})}{V(X_{n-2})} \frac{V(X_n)}{V(X_{n-1})} \one\{X_1 \in \mathcal U_1^c,\dots, X_n \in \mathcal U_1^c\} \right ]\\
&
=
\E_x \left\{ \E\left[ \frac{V(X_1)}{V(x)}\cdots \frac{V(X_{n-1})}{V(X_{n-2})} \frac{V(X_n)}{V(X_{n-1})} \one\{X_1 \in \mathcal U_1^c,\dots, X_n \in \mathcal U_1^c\} \ \Bigg|\  X_1,\dots,X_{n-1} \right ] \right\}
\\
&
\le
\E_x\left[ \frac{V(X_1)}{V(x)}\cdots \frac{V(X_{n-1})}{V(X_{n-2})} \one\{ X_1 \in \mathcal U_1^c,\dots, X_n \in \mathcal U_1^c\} \right ]
\cdot
{ \sup_{x' \in \mathcal U_1^c} \E \big( \textstyle \frac{V(X_n)}{V(X_{n-1})} \big| X_{n-1}=x' \big) }
\\
&
\le
\rho \cdot \E_x	\left[ \frac{V(X_1)}{V(x)}\cdots \frac{V(X_{n-1})}{V(X_{n-2})} \one\{ X_1 \in \mathcal U_1^c,\dots, X_n \in \mathcal U_1^c\} \right ]
\\
& \le
\cdots
\le
\rho^{n-1}
\cdot
\E_x \left[ \frac{V(X_1)}{V(x)} \right]
\le
\rho^{n-1}
\cdot
\sup_{x\in\Lambda} \E_{x} \left [\frac{V(X_1)}{V(x)} \right]
\le C \rho^{n-1}.
\end{align*}
By means of Theorem \ref{thm:h1h2h3}, we get that discrete time chain $(X_n)$ is $R-$ positive, which is equivalent to the $R-$positivity of its transition probability matrix ${\rm exp}(A)$.  Finally, observe that there is a bijection between the eigenvalues of $A$ and ${\rm exp}(A)$ and that the eigenvectors are exactly the same. This gives us the $R-$positivity of $A$, or equivalently, the process $(X_t)_{t\ge0}$.
\end{proof}
Next we give a different condition that guarantees $R$-positivity. It also has to be compared with \cite[Remark 11]{champagnat-villemonais} and \cite[Theorem 1.4]{FM}.
Define 
\begin{equation*}
a(z)= \inf_{x\in \Lambda\setminus\{z\}} a(x,z), \quad \alpha=\sum_{z\in\Lambda}a(z), \quad C=\sup_{x\in \Lambda} a(x,0). 
\end{equation*}

\begin{theorem}
\label{alpha.C}
If $\alpha>C$, then $A$ is $R$-positive. Moreover, the right eigenvector $\mu$ verifies $ 0< c_1 \le \mu(x) \le C_1$ for some positive constants $c_1$ and $C_1$ and all $x\in \Lambda$. As a consequence, $\nu$ is summable.
\end{theorem}
\begin{remark}
Under this condition, Ferrari and Mari\'c \cite{FM} proved existence of a quasi-stationary distribution (QSD) and Jacka and Roberts uniqueness \cite{jacka}. Here we prove $R$-positivity, which also leads to existence of such a QSD.
\end{remark}

\begin{proof}
The argument is reminiscent to the one in \cite[Theorem 4.1]{champagnat-villemonaisPTRF}, although applied in a different situation. Let $(X_t, t\ge 0)$ be a Markov process with transition rates $A$. By \cite[Theorem 1]{Kingman}, the parameter $\gamma$ defined by
\[
\lim_{t\to \infty} -t^{-1}\log \P_x(X_t=y) = \gamma, 
\]
is well defined, independent of $x,y$ and verifies for every $x\in \Lambda,\, t\ge0$,
\[
 -t^{-1}\log \P_x(X_t=x) \ge \gamma.
\]
Observe that for every $x\in \Lambda$ we have $\P_x(X_t\ne0) \ge e^{-Ct}$. Hence
\[
 \gamma \le  -t^{-1}\log \P_x(X_t=x) \le -t^{-1}\log \P_x(X_t \ne 0) \le C.
\]
Without loss of generality, we can assume that there exists a finite set $\mathscr U_1\subset \Lambda$ such that  
$\sum_{y\in \mathscr U_1}\inf_{x\in\Lambda} a(x,y)=\alpha>C.$ Then, for $R=e^{-\gamma}$, we can take $\rho>e^{-\alpha}$ such that 
$$\rho = e^{-\alpha} < e^{-\gamma} = R.$$
Let us check that $P:=\exp(A)$ verifies the hypothesis of Theorem \ref{thm:h1h2h3}. Observe that since $\P_x( \tau_{ \mathscr U_1}>n)=\P_x(X_k\notin \mathscr U_1, \mbox{for } k\in\{1,2,\ldots, n\})\leq e^{-\alpha n} $, we have
$$
\P_x(\tau_{\mathscr U_1}>n, \tau_0>n)\leq \P_x(\tau_{\mathscr U_1}>n)\leq e^{-\alpha n} = \rho^n.
$$
Then, $P$ is $R$-positive and there exist left and right eigenvectors $\nu$, $\mu$ respectively with eigenvalue $1/R$. By \cite{jacka}, $\nu$ is the only QSD of $P$. Also, $\sum_{x\in \Lambda}\nu(x)<\infty$ and by \cite[Theorem 3.1]{seneta-vere-jones-66} the following limits hold  
$$\lim_{n\to\infty}\P_x(X_n=y|\tau_0>n)=\nu(y), \qquad \lim_{n\to\infty}R^n\P_x(\tau_0>n)=\mu(x). $$
Next, we compute the bounds for $\mu$. Notice that 
 \begin{align*}
 \P_x(\tau_0 >n) & =   \P_x(\tau_0>n|\tau_{\mathscr U_1}>n)\P_x(\tau_{\mathscr U_1}>n) + \P_x(\tau_0 >n|\tau_{\mathscr U_1}\le n) \P_x(\tau_{\mathscr U_1}\le n)\\ 
               & \le   \P_x(\tau_{\mathscr U_1}>n) + \E_x(\P_{X_{{\tau_{\mathscr U_1}}}}(\tau_0 > n- \tau_{\mathscr U_1})\mathbf{1}\{\tau_{\mathscr U_1}\le n\})\\
               & \le e^{-\alpha n} + \E_x ( g(n-\tau_{\mathscr U_1})),
\end{align*}
where  $g(k)=\max_{y\in \mathscr U_1} \P_y(\tau_0 > k)$. Then, 
\begin{align*}
\mu(x)=\lim_{n\to \infty} R^{-n} \P_x(\tau_0 >n) & \le \limsup_n \left[ R^{-n}e^{-\alpha n} + \E_x( R^{-n}g(n-\tau_{\mathscr U_1}))\right]\\
 & \le  \limsup_n \E_x \left( R^{-(n-\tau_{\mathscr U_1})} g(n-\tau_{\mathscr U_1}) R^{-\tau_{\mathscr U_1}} \right)\\
 & \le  \left (\sup_n R^{-n} g(n)\right ) \E_x R^{-\tau_{\mathscr U_1}}=:C_1.
\end{align*}
The first factor is finite since $R^{-n}g(n)$ has a limit which does not depend on $x$. To bound the second factor, observe that since  
$\P_x(\tau_{\mathscr U_1}>n) \le e^{-\alpha n}=\rho^{n}$
for all  $n \ge 1$, we have  
$$\E_x( R^{-\tau_{\mathscr U_1}})=\sum_{n\in\N}R^{-n}\P(\tau_{\mathscr U_1}=n)\le \sum_{n\in\N}R^{-n}\P(\tau_{\mathscr U_1}>n-1)=\frac{1}{R}\sum_{n\in\N_0}R^{-n}\rho^{n} <\infty.$$
For the lower bound we compute,
\begin{align*}
 \mu(x)  & = \gamma\sum_{y\in\Lambda} a({x,y})\mu(y) \ge \gamma\sum_{y\in\Lambda} \inf_{x\neq y} a(x,y)\mu(y)  \ge \gamma\sum_{y\in\mathscr U_1} \inf_{x \neq y } a({x,y})\mu(y)\\
 & \geq 
 \gamma\alpha \min_{y \in\mathscr U_1} \mu(y) 
 =: c_1 >0.
\end{align*}
\end{proof}

\section{Tumor growth and multitype branching}
\label{application}
In this section we will address the asymptotic behavior of $(\eta_t)$ as time goes to infinity in the supercritical regime. Our proofs will strongly rely on a theorem by Moy (a version of Kesten-Stigum theorem for countable types) for the behavior of supercritical multitype branching processes.
We state the theorem for completeness. We will slightly abuse notation by using $\E_y$ and $\P_y$ to denote that the process $Z$ starts with one particle at $y$ at $n=0$. We use $|Z_n| $ for the total number of particles in the sytem at time $n$, that is
 $|Z_n| = \sum_{y \in \Lambda} Z_n (y)$.

\begin{theorem}[{\cite[Theorem 1]{moy}}]\label{thm.moy} Let $(Z_n)_{n\ge 0}$ be a multitype branching process with a countable types space $\Lambda$ and mean matrix $M$ with entries $m(x,y)=\E_x(Z_1(y))$, $x,y \in \Lambda$. Assume $M$ is aperiodic, irreducible and $R$-positive with normalized left and right eigenvectors $\nu,\, \mu$ respectively, $R>1$ (supercritical) and
\begin{equation}
\label{condition.moy}
   \sum_{y\in \Lambda}\E_y((\sum_{x\in \Lambda}Z_1(x) \mu(x))^2)\nu(y)<\infty.
\end{equation}
Then there is a real valued random variable $W$, with $\E(W^2)<\infty$ such that for every $f$ with $\sup_x f(x)/\mu(x) <\infty$ we have
\[
\lim_{n\to \infty}\E_x\left[ R^{-n}\sum_{y\in \Lambda}Z_n(y)f(y) - W\sum_{y\in\Lambda}\nu(y)f(y)\right]^2=0.
\]
In particular, for every $y\in \Lambda$, $\E_x[R^{-n}Z_n(y) - W\nu(y)]^2 \to 0$ and if $\inf_{x\in\Lambda}\mu(x)>0$, $\E_x[R^{-n}|Z_n| - W]^2 \to 0$.
\end{theorem}
As a consequence, we obtain the following corollary, as in \cite{JS}
\begin{corollary}
\label{true.empirical.measure}
In the conditions of Theorem \ref{thm.moy}, assume that for every $x \in \Lambda$ and $Z_0=\delta(x,\cdot)$, conditioned on non-extinction, there is $y\in \Lambda$ with $Z_n(y) \to \infty$ in probability as $n\to \infty$, then $Z_n/|Z_n| \to \nu$ in probability on this event. 
\end{corollary}

Before embarking on the proof, we need to introduce some notation. Define the moment generating function $\ff\colon [0,1]^\Lambda \to \R^{\Lambda}$ by
\[
\ff(\ss) = (f_x(\ss), \,\, x\in \Lambda)= \left(\E_{x} \left(\prod_{y\in\Lambda} s_y^{Z_1(y)}\right), \,\, x\in\Lambda \right). 
\]
If we define 
\[
\ff_n(\ss) = (f_{nx}(\ss), \,\, x\in \Lambda)= \left(\E_{x} \left(\prod_{y\in\Lambda} s_y^{Z_n(y)}\right), \,\, x\in\Lambda \right),
\]
we get, as for single-type Galton-Watson processes, $\ff_{n+1}(\ss)= \ff(\ff_n(\ss))$. We also define $\qq:=(q_x, \, x\in\Lambda)$, with $q_x$ being the absorption probability when the process starts with one individual of type $x$, i.e. $q_x:=\P_x\left(\bigcup_{n\in\N}\{ Z_n=\mathbf 0\}\right)$.
 
\begin{proof}The proof follows mainly \cite{JS}. The function $\ff$ has at least two fixed points: $\uno=(1, x\in\Lambda)$ and $\qq$. It is also known that if $\ss$ is a fixed point of $\ff$ different from $\uno$, then we have $s_x<1$ for every $x\in \Lambda$.

For the first claim, $\uno$ is clearly a solution to $\ff(\ss)=\ss$ and for the extinction event, we have 
$$
q_x=\P_x\left(\bigcup_{n\in\N}\{Z_n=\mathbf 0\}\right)=\lim_{n\to\infty}\P_x(Z_n=\mathbf 0)=\lim_{n\to\infty}f_{nx}(\mathbf 0). 
$$
That is, $\qq=\displaystyle\lim_{n\to\infty}\ff_n(\mathbf 0)$. Since $\ff_{n+1}(\mathbf 0)=\ff(\ff_n (\mathbf 0))$ and $\ff$ is continuous in this topology, we get $\qq=\ff(\qq)$.

For the second claim, let $\ss\in [0,1]^{\Lambda}$ be a fixed point of $\ff$ with $s_{x}<1$ for some $x$ and assume there is $y\in\Lambda$ with $f_{y}(\ss)=1=s_y$. Then $\ss$ is a fixed point of $\ff_n$ for every $n$.
Let $\tt\in [0,1]^{\Lambda}$ be such that 
 $$t_{w}=1 \ \ \mbox{if}\ w\neq x$$
 $$t_{w}=s_{x}\ \ \mbox{if}\ w=x.$$
Since $\ss\leq\tt$, $f_{ny}(\ss)=1$, and the function $f_{ny}$ is monotone increasing, we get $1=f_{n y}(\tt)$ and hence 
$$
1=\sum_{\ii\in \N^{\Lambda}}\P_{y}(Z_n=\ii) \tt^{\ii}=\sum_{\ii\in \N^{\Lambda}}\P_{y}(Z_n=\ii)  s_{x}^{z_{x}}.
$$ 
Then $\P_{y}(Z_n=\ii)=0$ if $z_{x}>0$, since $\sum_{\ii\in \N^{\Lambda}}\P_{y}(Z_n=\ii)=\uno.$ So,  
\[
m^{(n)}(y,x)=\E_{y}(Z_n(x))=\sum_{\ii\in \N^{\Lambda}}\P_{y}(Z_n=\ii) z_{x}=0,
\]
for every $n\in\N$. This contradicts the irreducibility of $M$. Then $\ss<\uno$.

Let $\ss$ be a fixed point of $\ff$ different from $\uno$. We have $s_x<1$ for every $x\in \Lambda$. Denote $\Omega_{\rm surv}=\{|Z_n| \to \infty \}= \{|Z_n|\ne 0\, \text{ for very } n\ge 0\}$. Let $y$ be such that $Z_n(y)\to \infty$ in probability on $\Omega_{\rm surv}$, then
\begin{align*}
s_x & = f_{n x}(\ss)=\E_x\left(\prod_{w\in\Lambda}s_w^{Z_n(w)}\Big| \, \Omega_{\rm surv}\right)\P_x(\Omega_{\rm surv}) + \E_x\left(\prod_{w\in\Lambda}s_w^{Z_n(w)}\Big|\, \Omega_{\rm surv}^c\right)\P_x(\Omega_{\rm surv}^c).\\
\end{align*}
By dominated convergence we get 
$$
\E_x(s_y^{Z_n(y)}\prod_{w\neq y }s_w^{Z_n(w)} | \,\Omega_{\rm surv}) = \E_x(s_y^{Z_n(y)}\prod_{w\neq y }s_w^{Z_n(w)} | \,\Omega_{\rm surv})\leq \E_x(s_y^{Z_n(y)} | \,\Omega_{\rm surv}) \to 0,
$$ 
and
\[
 \E_x\left(\prod_{w\in\Lambda}s_w^{Z_n(w)}\Big|\, \Omega_{\rm surv}^c\right) \to 1
\]

and hence $s_x=q_x$ for every $x\in\Lambda$. We conclude that the only fixed points are $\qq$ and $\uno.$

Let $\rr=(r_x, \,\, x\in \Lambda)$ be the vector with coordinates $r_x=\P_x(W=0)$. It turns out that $\rr$ is a fixed point of $\ff$. To see that, we compute for $x\in\Lambda$, 
\begin{align*}
 r_x&=\P_x(W=0)=\sum_{\jj\in\N^{\Lambda}}\P_x(W=0|Z_1=\jj)\P_x(Z_1=\jj) \\
& =\sum_{\jj\in\N^{\Lambda}}\P_x(Z_1=\jj) \prod_{y\in \Lambda }\P_y(W=0)^{z_y}\\
&=f_x(\P(W=0))=f_x(\rr).
 \end{align*}

As a consequence, we get the dichotomy $\rr=\qq$ or $\rr=\uno$. Now, by Theorem \ref{thm.moy}, Taking  $f=\delta_y$,  we get
$$ \lim_{n\to\infty}\E_x((R^{-n}Z_n(y)-\nu(y)W)^2)=0,$$ and taking $f\equiv1$ and using $\inf\mu(y)>0$, we get
$$ 
\lim_{n\to\infty}\E_x((R^{-n}\sum_{y\in\Lambda}Z_n(y)-W)^2)=0.  
$$
Then $\lim_{n\to\infty}R^{-n}|Z_n|=W$ in probability and hence $\{|Z_n|\to 0\}\subseteq \{W=0\}$. Since both events have the same probability we get
\[
 \P_x(W=0|\, \Omega_{\rm surv})=0.
\]
This fact allow us to compute (on $\Omega_{\rm surv}$) the following limit in probability,
$$\lim_{n\to\infty} \frac{Z_n(y)}{|Z_n|}=\lim_{n\to\infty} \frac{R^n Z_n(y)}{R^n\sum_{y\in\Lambda}Z_n(y)}= \frac{\nu(x) W}{W}  =\nu(x).$$ 
\end{proof}

\subsection{Identification with a multitype branching process}
The process $(\eta_t)$ defined by \eqref{generator} can be identified with a multitype branching process in the following way.
Recall that $q(x,x)=-\sum_{y\ne x, y\ne 0} q(x,y)$ and $q(x)=-q(x,x)$. An individual of type $x\in \Lambda$ gives birth (and die) at an exponential time of parameter $a(x):=\beta(x)+q(x)$ and the offspring distribution is given by
\begin{itemize}
 \item with probability $\frac{\beta(x)}{a(x)}$, one child of type 1 and one child of type $x$.
 \item with probability $\frac{q(x,y)}{a(x)}$, one child of type $y$.
\end{itemize}
According to this, the mean matrix $M=(m(x,y), \, x,y \in \Lambda)$  for the skeleton chain is given by
\[
 m(x,y)= \begin{cases}
		  \frac{q(x,y)}{a(x)} & y \notin\{0,1,x\},\\
		  \frac{\beta(x)+q(x,1)}{a(x)} & y=1,\\
		  \frac{\beta(x)}{a(x)} & y=x.
               \end{cases}
\]
The expected number of individuals  for the continuous time process at time $t$ is given by the matrix $\exp(tA)$, where $A$ has entries
\[
 a(x,y)=a(x)(m(x,y)-\delta(x,y)).
\]
Here $\delta(x,y)=1$ if $x=y$ and $0$ otherwise.
\begin{proposition}\label{prop.R.positive}
Assume $\beta$ is bounded above and one of the following conditions is verified.
\begin{enumerate}
 \item There is $V\colon \Lambda \to \R_+$ such that 
 \[
\lim_{x\to\infty}V(x)=\infty, \qquad \lim_{x\to \infty}  \frac{\sum_{y\ge 0} q(x,y)V(y)}{V(x)}= -\infty.
\] 
 \item $\sup_{x\in\Lambda}\beta(x) - \inf_{x\in\Lambda}(\beta(x)-q(x,0)) < \inf_{x\in\Lambda}\beta(x)$.
\end{enumerate}
Then, $A$ is $R$-positive with right eigenvector $\mu$ and left eigenvector $\nu$ such that $\sum_{x\in \Lambda} \mu(x)\nu(x) <\infty$. Moreover, $\sum_{x\in \Lambda} \nu(x) <\infty$.
\end{proposition}
As a consequence, we obtain the following corollaries.

\begin{corollary}\label{moy.alpha.C} In the notation of Proposition \ref{prop.R.positive}, if condition (2) is verified, $Q$ has a Yaglom limit (i.e. $\exp(tQ)(x,y)/(1-\exp(tQ)(x,0)) \to \nu(y)$ as $t \to \infty$ for every $x\in \Lambda$) and $\kappa_0>1$, then \eqref{lim.proportions} holds.
\end{corollary}

\begin{corollary}\label{moy.birth.death} In the notation of Proposition \ref{prop.R.positive}, if $Q$ is the rates matrix of a one dimensional birth and death process that verifies,
\begin{enumerate}
 \item[(a)] $(q(x,x-1))_{x\in \N}$ is monotone increasing, \label{uno}
 \item[(b)] $\lim_{x\to\infty} \frac{q(x,x+1)}{q(x,x-1)} = \ell <1,$ \label{dos}
 \item[(c)] $\sum_{x\in\N} \frac{1}{q(x,x-1)} <\infty.$ \label{tres}
\end{enumerate}
and $\kappa_0>1$, then \eqref{lim.proportions} holds.
\end{corollary}

\begin{proof}[Proof of Proposition \ref{prop.R.positive}]
Let $\bar \beta = \sup_x\beta(x)$ and consider the matrix $\tilde A$ with coefficients $\tilde a(x,y)=a(x,y)- \bar\beta\delta(x,y)$. Since $\sum_{y\ne x} \tilde a(x,y) \le 0$ for every $x\in \N$, we can think of $\tilde A$ as the rates matrix of a process absorbed at zero that we call $X=(X_t, t\ge 0)$. Observe that the absorption rate from state $x$ is given by $-\sum_{y\ne x} \tilde a(x,y)$. We are going to prove that $X$ is $R$-positive and hence the same holds for $\exp(tA)$ for every $t >0$. 
Assume condition {\em 1.} is verified, then we compute for this chain, for $x\ne 0, 1$, 
\begin{align}
\label{lyapunov.r.positive}
 \dot V(x) &= \sum_{y\in \Lambda0} \tilde a(x,y)V(y)= \sum_{y\ne x} q(x,y)V(y) + \beta(1)V(1) - (\bar \beta +a(x) -\beta(x))V(x)\\
 & = \sum_{y\in \Lambda_0} q(x,y) V(y) + \beta(1)V(1) - \bar \beta.
\end{align}
Hence, we have $\dot V(x)/V(x) \to -\infty$ and $V(x)\to \infty$ as $x\to \infty$ and we can apply Theorem \ref{Lyapunov} to get the $R$-positivity of $\tilde A$, $A$ and ${\rm exp}(tA)$ for every $t >0$.

If condition {\em 2} is verified, we apply Theorem \ref{alpha.C} instead. We consider again the Markov process with rates matrix $\tilde A$. For this matrix we have $ \alpha(1) = \inf_x \tilde a(x,1) \ge  \inf_x {\beta(x)}$ and
\[
C=\sup_x\tilde a(x,0) \le  \sup_x q(x,0) - \beta(x) + \bar \beta = \sup_{x\in\Lambda}\beta(x) - \inf_{x\in\Lambda}(\beta(x)-q(x,0)).
\]
Then, we can apply Theorem \ref{alpha.C}.

The limit \eqref{lim.proportions} will be a consequence of Theorem \ref{thm.moy} and some additional considerations. Once $R-$ positivity is proved, we need to check \eqref{condition.moy}. This is a delicate condition which is not simple to prove in general. We will prove that in fact \eqref{condition.moy} is verified under our hypotheses. Next, we also need to show that Corollary \ref{true.empirical.measure} can be applied to get \eqref{lim.proportions}.

\begin{proof}[Proof of Corollary \ref{moy.alpha.C}]

Under condition (2) we have shown in Theorem \ref{alpha.C} that $\mu$ is bounded above and then \eqref{condition.moy} reduces to 
\[
 \sum_{y\in \Lambda}\E_y|Z_1|^2\nu(y) <\infty.
\]
Observe that the total number of particles in the system at time $t$ is stochastically dominated (uniformly in $y$) by a continuous time Galton-Watson process with binary branching and reproduction rate $\sup_x \beta(x)$, which has a second moment uniformly bounded in $y$. Since $\sum_x \nu(x) <\infty$, \eqref{condition.moy} holds in this case and we get for $y\in \Lambda$,
\[
\E_x[R^nZ_n(y) - W\nu(y)]^2 \to 0.
\]
Since $\inf_x\mu(x)>0$, we also get
\[
\E_x[R^n|Z_n| - W]^2 \to 0.
\]
Now we prove that for every $x \in \Lambda$ there exists a $y \in \Lambda$ such that $P_x(Z_n(y) \to \infty)=1$. Observe that, with the exception of the ancestral particle, every other particle is born at one and hence, if it is still in the system (has not been absorbed yet), the probability of being at one is (for some $t$) $\exp(tQ)(1,1)/(1-\exp(tQ)(1,0))$, the conditional probability of being at one given that it was not absorbed, which has a positive limit and is positive for every finite time due to the existence of Yaglom limit for $Q$. Hence its infimum $c$ is larger than zero. We have for every $K$ 
\[
 \P_x(Z_n(1)\le K|\, \Omega_{\rm surv}) \le \E_x[|Z_n|^K(1-c)^{|Z_n|-K}|\, \Omega_{\rm surv}].
\]
By bounded convergence theorem we have 
\[
\lim_{n\to\infty} (1-c/2)^{-n}\E_x[|Z_n|^K(1-c)^{|Z_n|-K}|\,\Omega_{\rm surv}]=0,
\]
and using Borel-Cantelli's Lemma we get $Z_n(1)\to \infty$ a.s. in $\Omega_{\rm surv}$ and by Corollary \ref{true.empirical.measure} 
\[
\frac{\eta_n(x)}{|\eta_n|} \to \nu(x), \qquad \text{in probability for every } x\in \Lambda.
\]
To go from discrete to continuous time we use \cite[Theorem 2]{kingman-63a}. The following computations are conditioned on $\Omega_{\rm surv}$. Given $\epsilon>0,$ for every $y\in\Lambda$ we consider the function $g_y\colon [0,\infty) \to [0,1]$ defined by
$$
g_y(t)=\P_x\left(\left|\frac{\eta_t(y)}{|\eta_t|}-\nu(y)\right|>\epsilon\right).
$$ 
We show that $g_y$ is continuous. Let $0\leq s<t$. We have,
\begin{align*}
\Big|\P_x\left(\left|\frac{\eta_t(y)}{|\eta_t|} - \nu(y)\right|>\epsilon\right) & -\P_x\left(\left|\frac{\eta_s(y)}{|\eta_s|}-\nu(y)\right|>\epsilon\right)\Big|
\\ & \le \E_x\left(\left|\mathbf{1}\left\{\left|\frac{\eta_t(y)}{|\eta_t|}-\nu(y)\right|>\epsilon\right\}-
\mathbf{1}\left\{\left|\frac{\eta_s(y)}{|\eta_s|}-\nu(y)\right|>\epsilon\right\}\right|\right)\\
 & \le \E_x\left(\mathbf{1}\left\{\frac{\eta_t(y)}{|\eta_t|}\neq\frac{\eta_s(y)}{|\eta_s|}\right\}\right)\\
 &\le \E_x(1-e^{-|\eta_s|(t-s)}).
\end{align*}
By dominated convergence theorem we get the continuity of $g_y(t)$. Observe that if instead of considering the process $\eta_t$ at times $t=0,1,2, \dots$ we would have been considered it at times $t=0, \delta, 2\delta, \dots$ we would have been obtained exactly the same result.
That is, $g_y(n\delta) \to 0$ as $n\to \infty$, for every $\delta >0$. By \cite[Theorem 2]{kingman-63a}, we have the convergence in $t$,
$$
\lim_{t\to\infty}g_y(t)=\lim_{t\to\infty}\P_x\left(\left|\frac{\eta_t(y)}{|\eta_t|}-\nu(y)\right|>\epsilon\right)= 0,
$$
proving the result \eqref{lim.proportions}.
\end{proof}

\begin{proof}[Proof of Corollary \ref{moy.birth.death}]
We are going to prove that $\mu$ is bounded. So first, consider the process with rates $\tilde A$ as in the proof of Proposition \ref{prop.R.positive} and take $V(x)=x$. For this process we have computed in \eqref{lyapunov.r.positive} 
\begin{align*}
\dot V(x) & = \sum_{y\ge 0} q(x,y) V(y) + \beta(1)V(1) - \bar \beta\\
&= -q(x,x-1) + q(x,x+1)  + \beta(1)V(1) - \bar \beta \\
& \le -(1-\ell)q(x,x-1)  + \beta(1)V(1) - \bar \beta,
\end{align*}
by (b). By (a) and (c) we have,
\[
 \frac{\dot V(x)}{V(x)} \le  x^{-1}(-(1-\ell)q(x,x-1)  + \beta(1)V(1) - \bar \beta) \to -\infty, \quad \text{as } x\to \infty.
\]
Hence, $\tilde A$ and $A$ are $R$-positive and Yaglom limit exists. Next, observe that (a), (b) and (c) implies that the process with rates $Q$ ``comes down from infinity'', meaning that $\sup_x\E_x(\tau_1) <\infty$, \cite[Lemma 2.2 and Proposition 2.2]{BMR}.

We need to find a bound for $\mu$.  Since $A$ is $R$-positive, $\mu$ has the following characterization
\begin{equation}
\label{char.mu}
\frac{\mu(x)}{\mu(1)}=\lim_{t\to \infty} \frac{\E_x|\eta_t|}{\E_1|\eta_t|}.
\end{equation}
To bound $\E_x|\eta_t|$ observe that if we start with one particle at $x$ (that we call the ancestral particle), this particles will produce during its whole life (before being absorbed) a random number of particles that can be bounded by the number of occurrences in a one-dimensional Poisson process with rate $\sup_x \beta(x)$ at an independent random time $T_x$ (the absorption time of the process with rates $Q$, started at $x$). We have just proved that the expectation of $T_x$ is bounded above by some constant $\kappa$. Each particle produced by the ancestral particle is born at $1$ and hence, its expected number of descendants by time $t$ is bounded by $\E_1|\eta_t|$. Plugging into \eqref{char.mu} we get
\[
\frac{\mu(x)}{\mu(1)}=\lim_{t\to \infty} \frac{\E_x|\eta_t|}{\E_1|\eta_t|} \le \lim_{t\to \infty} \frac{\E(T_x)\bar \beta\E_1|\eta_t|}{\E_1|\eta_t|} \le \kappa \bar \beta <\infty. 
\]
The rest of the proof is as in case 2 since we have shown $R-$positivity of $\tilde A$, existence of Yalgom limit for $\tilde A$ and boundedness of $\mu$.
\end{proof}

\begin{center}
{\bf Acknowledgments} 
\end{center}
We thank Pablo Ferrari for several years of enlightening conversations on this topic. P. Groisman and K. Ravishankar were supported by Simons
collaboration grant number 281207 awarded to K. Ravishankar. P. Groisman and A. Ferrari are partially founded by UBACYT 20020160100147BA
and PICT 2015-3154. The authors want to thank NYU-Abu Dhabi where part of this work was done during  P. Groisman's visit for hospitality and support.

\end{proof}

\end{document}